\documentclass[10pt]{article}
\usepackage{amssymb}
\usepackage{mathtools}
\usepackage{hyperref}
\usepackage{amsmath,amsthm,amssymb}
\usepackage{mathtext}
\usepackage[T1,T2A]{fontenc}
\usepackage[utf8]{inputenc}
\usepackage[english]{babel}
\usepackage{graphicx}
\graphicspath{{Article/}}
\DeclareGraphicsExtensions{.pdf,.png,.jpg}
\usepackage{titlesec}
\titleformat*{\section}{\normalsize \bfseries}
\newtheorem{theorem}{Theorem}
\newtheorem{lemma}{Lemma}

\makeatletter
\renewenvironment{proof}[1][\proofname]{%
	\par\pushQED{\qed}\normalfont%
	\topsep6\p@\@plus6\p@\relax
	\trivlist\item[\hskip\labelsep\bfseries#1\@addpunct{.}]%
	\ignorespaces
}{%
	\popQED\endtrivlist\@endpefalse
}
\makeatother

\begin{document}
	
\title{Polynomial approximation of piecewise analytic functions on quasi-smooth arcs}
\author{Liudmyla Kryvonos}
\date{}
\maketitle 
\thispagestyle{empty}

\begin{abstract} 
	For a function $f$ that is piecewise analytic on a quasi-smooth arc  $\mathcal{L}$ and any $0<\sigma<1$ we construct a sequence of polynomials that converge at a rate $e^{-n^{\sigma}}$ at each point of analyticity of $f$ and are close to the best polynomial approximants on the whole $\mathcal{L}$. Moreover, we give examples when such polynomials can be constructed for $\sigma=1$.
\end{abstract}



\section{Introduction and main results}

Let $\mathcal{L}$ be a quasi-smooth arc on the complex plane $\textbf{C}$, that is, for any $z$ , $\zeta \in \mathcal{L}$ the length $|\mathcal{L}(z, \zeta)|$  of the subarc $\mathcal{L}(z, \zeta)$ of $\mathcal{L}$ between points  $z$, $\zeta$ satisfies
$$
|\mathcal{L}(z, \zeta)| \leqslant c |z-\zeta|
$$
for some $c = c(\mathcal{L})\geqslant1$.

Consider a piecewise analytic function $f$ on $\mathcal{L}$ belonging to 
$C^{k}(\mathcal{L})$, $k \geq 0$, that means $f$ is $k$ times continuously differentiable on $\mathcal{L}$ and there exist 
points 
$
z_{2}, z_{3},...,z_{m-1}
$
such that $f$ is analytic on $\mathcal{L} \backslash \{z_{1}, z_{2},...,z_{m}\}$, ($z_{1}, z_{m}$ -- endpoints of $\mathcal{L}$), but is not 
analytic at points $z_{1}, z_{2},...,z_{m}$.
We call the $z_{i}$ points of singularity of $f$.
 
The rate of the best uniform approximation of a function $f$ by polynomials of degree  at most $n \in \mathbb{N} := \{1,2,...\}$ is denoted by
$$
E_{n}(f) = E_{n}(f, \mathcal{L}) := \underset{P_{n}: deg P_{n} \leq n}{inf}\|f - P_{n}\|_{\mathcal{L}}.                                                         \eqno(1.1)
$$
Here $\|\cdotp\|_{\mathcal{L}}$ means the supremum norm over $\mathcal{L}$. 
Also, let $p_{n}^{*}(f,z)$ be the (unique) polynomial minimizing the uniform norm in (1.1). 

It is natural to expect the difference $f(z) - p_{n}^{*}(z)$ to converge faster at points of analyticity of $f$. But, it turns out, singularities of $f$ adversely affect the behavior over the whole $\mathcal{L}$ of a subsequence of the best polynomial approximants $p_{n}^{*}(f, z)$. This so-called "principle of contamination" manifests itself in density of extreme points of $f - p_{n}^{*}$, discussed by A. Kroo$^{'}$ and E.B. Saff in \cite{Kroo} and accumulation of zeros of $p_{n}^{*}(f, z)$, showed by H.-P. Blatt and E.B. Saff in \cite{Blatt}.
For more details, we refer the reader to \cite{Saff}.
 
 Surprisingly, such behavior of zeros and extreme points need not hold for polynomials of "near-best" approximation, that is for polynomials $P_{n}$ that satisfy
$$
\|f - P_{n}\|_{\mathcal{L}} \leq C E_{n}(f), \qquad n=1,2,...,
$$  
with a fixed $C>1$.
Hence, it is natural to seek "near-best" polynomials which would converge faster at points $z \in \mathcal{L} \backslash \{z_{1}, z_{2},...,z_{m}\}$.

For the case of $\mathcal{L} = [-1,1]$ and a piecewise analytic function $f$ belonging to $C^{k}[-1,1]$, E.B. Saff and V. Totik in \cite{AppPieceFunct} have proved that if 
 non-negative numbers $\alpha, \beta$ satisfy $\alpha < 1$ and $\beta \geqslant \alpha$ or $\alpha=1$ and $\beta > 1$, then there exist constants $c$, $C>0$ and polynomials $P_{n}$, $n=1,2,...,$ such that for every $x\in[-1,1]$
$$
|f(x) - P_{n}(x)| \leqslant C E_{n}(f) e^{-c n^{\alpha} d(x)^{\beta}},         \eqno(1.2)
$$
where $d(x)$ denotes the distance from $x$ to the nearest singularity of $f$ in $(-1,1)$. 

Accordingly, the question of constructing "near-best" polynomials arises when $[-1,1]$ is replaced by an arbitrary quasi-smooth arc $\mathcal{L}$ in $\textbf{C}$. Polynomial approximation of functions on arcs is an important case of a more general problem of approximation of functions on an arbitrary continuum of the complex plane studied in the works of N.A. Shirokov \cite{ShirokovDocl}, V.K. Dzjadyk and G.A. Alibekov \cite{Alibekov}, V.V. Andrievskii \cite{AppOnArcs} and others (see, for example, \cite{Dzyadyk}).

The behavior of "near-best" polynomials is well studied in the case of approximation on compact sets $K$ with non-empty interior $Int(K)$. The following results demonstrate how the possible rate of convergence inside $K$ depends on the geometry of $K$. V.V. Maimeskul have proved in \cite{Maimeskul} that if $\Omega:=\overline{C}\setminus K$ satisfies the $\alpha$-wedge condition with $0<\alpha \leqslant 1$, then for any $\sigma < \alpha/2$ there exist "near-best" polynomials converging at a rate $e^{-n^{\sigma}}$ in the interior of $K$. E.B. Saff and V.Totik in \cite{BehaviorBestUnApp} show the possibility of geometric convergence of "near-best" polynomials inside $K$ if the boundary of $K$ is an analytic curve. 
Meantime, N.A. Shirokov and V. Totik in \cite{AppOnTheBoundAndInside} discuss the rate of approximation by "near best" polynomials of a function $f$ given on a compact set $K$ with a generalized external angle smaller than $\pi$ at some point $z_{0}\in \partial K$. They showed that if $f$ has a singularity at $z_{0}$, then geometric convergence inside $K$, where $f$ is analytic, is impossible.
Taking into account these results, the most interesting case for us is when singularities of the function $f$ occur at points where the angle between subarcs of $\mathcal{L}$ is different from $\pi$. It turns out that for some such arcs there are no restrictions on the rate of convergence of "near-best" polynomials and it can be geometric at points where $f$ is analytic,  as opposed to the result for compact sets with non-empty interior. We formulate and prove this assertion in Theorem 2.
Furthermore, the general case is given by the following

\begin{theorem}

Let $f$  be a piecewise analytic function on a quasi-smooth arc $\mathcal{L}$, i.e. there exist points $z_{2}, ... , z_{m-1} \in  \mathcal{L}$, such that they divide $\mathcal{L}$ into $\mathcal{L}^{1}, \mathcal{L}^{2}, ... , \mathcal{L}^{m-1}$ and 
$$
f(z) = f_{i}(z),   \: z \; \in \mathcal{L} ^{i},   \quad  i=\overline{1,m-1},   \eqno(1.3)
$$
where $f_{i}(z)$ are analytic in some neighborhood of $\mathcal{L} ^{i}$, respectively, and satisfy
$$
f_{i-1}^{(r)}(z_{i})=f_{i}^{(r)}(z_{i}) , \quad f_{i-1}^{(k_{i})}(z_{i}) \neq f_{i}^{(k_{i})}(z_{i})   \eqno(1.4)
$$
for $r=\overline{0,k_{i}},\; i=\overline{2,m-1}$.
Then, for any $0 < \sigma < 1$, there exists a sequence  $\{P_{n}\}_{1}^{\infty}$ of "near-best" polynomial approximants of $f$ on $\mathcal{L}$, such that 
$$
\underset{n \longrightarrow \infty}{lim} \|f - P_{n}\|_{E} \; e^{n^{\sigma}} = 0       \eqno(1.5)
$$
holds for any compact set $E \subset \mathcal{L} \backslash \{z_{2},...,z_{m-1}\}$. 
\end{theorem}
 On the complex plane consider lemniscates that are level lines of some complex polynomials. Namely, take $P(z)=P_{N}(z):=(z-a_{1})(z-a_{2})...(z-a_{N})$, where $a_{k}=Re^{i \frac{ 2 \pi (k-1)}{N}}$, $k=\overline{1,N}$ and $R>0$ is a fixed number. Then $|P(z)|=R^{N}$ is an equation of a lemniscate. Note that the origin is a point of this lemniscate (since |$P(0)|=R^{N}$).
 
 The lemniscate divides the plane into three parts, namely the curve itself, points $\{z: |P(z)|<R^{N}\}$ and $\{z: |P(z)|>R^{N}\}$.
 Consider an arc $\mathcal{L}=\mathcal{L'} \cup \mathcal{L''}$, where  $\mathcal{L'}, \mathcal{L''}$ may belong to different petals of the lemniscate, meet at the origin and satisfy $|P(z)|<R^{N}$,  $ z \in \mathcal{L} \setminus \{0\}$.
 An example for $N=4$, $R=1$ you can see below.
 
 \begin{figure}[h!]
 	\center{\includegraphics[scale=0.6]{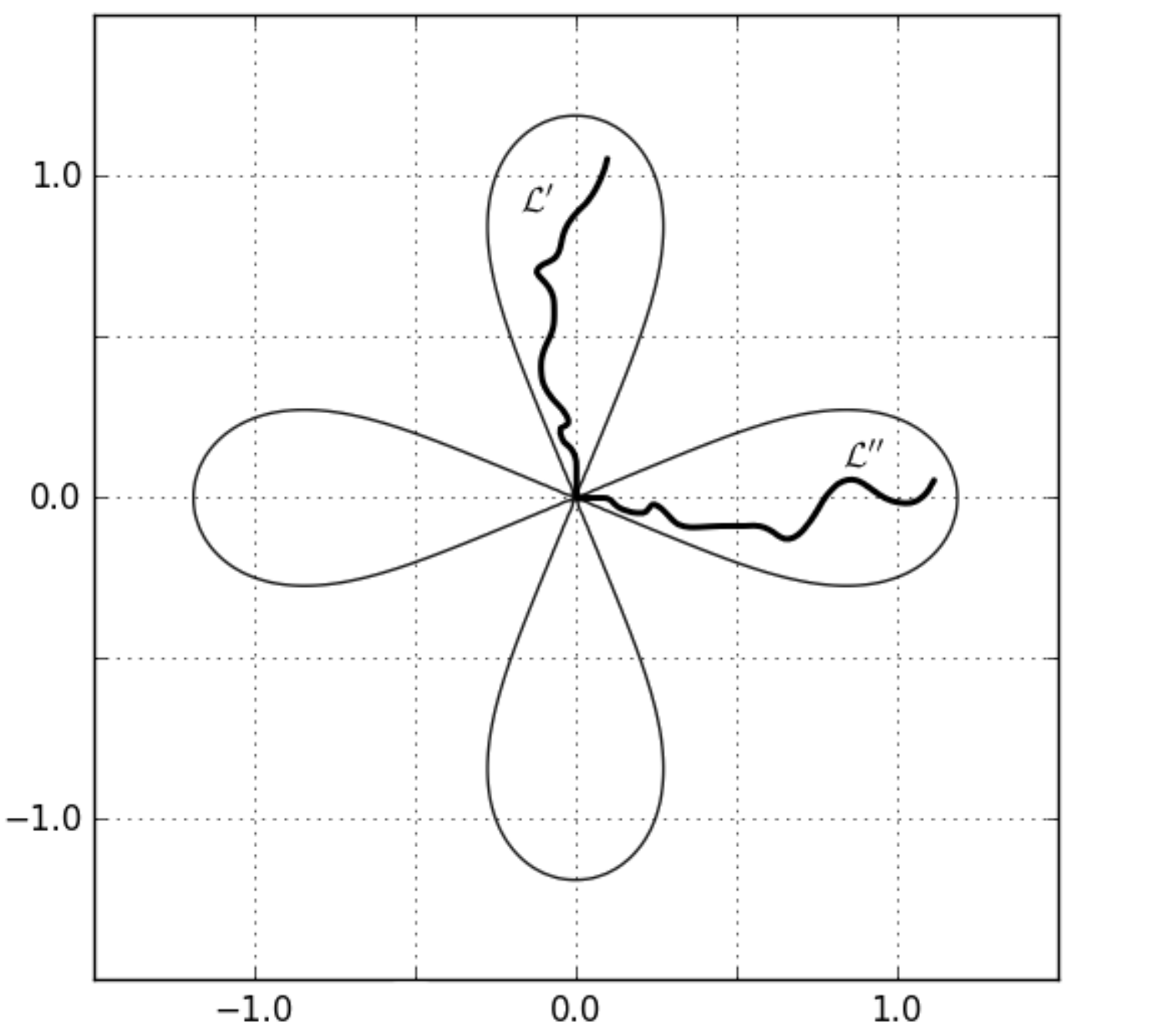}} 
 \end{figure}
 
 In particular, two line segments meeting at the origin at angle $0< \varphi \leqslant \pi$ satisfy this property: if $\frac{2 \pi}{m+1}<\varphi \leqslant \frac{2 \pi}{m}$ for some integer $m$, it is enough to take $R$ to be sufficiently large and $N=m$.

 Let $f$ be a piecewise analytic function on $\mathcal{L}$ given by
 $$ 
 f(z)=\begin{cases}                                                              
 f_{1}(z),&\text{if $z$ $\in \mathcal{L} ^{'}$}\\                             
 f_{2}(z),&\text{if $z$ $\in \mathcal{L} ^{''}$}
 \end{cases}                                                                   
 $$
 where $f_{1}$, $f_{2}$ are functions, analytic on $\mathcal{L} ^{'}$ and $\mathcal{L} ^{''}$ correspondingly, satisfying
 $$
 f_{1}^{(r)}(0)=f_{2}^{(r)}(0)  , \; r=\overline{0,k},    \quad f_{1}^{(k+1)}(0)\neq f_{2}^{(k+1)}(0).      
 $$
 With these assumptions we prove the following result

\begin{theorem} 
Let $\mathcal{L}$ and $f$ be as above.
Then there exist a constant $c>0$ and a sequence of "near-best" polynomials $\{P_{n}\}_{1}^{\infty}$,
such that
$$
\underset{n \longrightarrow \infty}{lim} \|f - P_{n}\|_{E} \; e^{c n d(E)} = 0,
$$
where $d(E)>0$ for any compact set $E \subset \mathcal{L}\setminus \{0\}$.
\end{theorem}

\section{Auxiliary results}

   In this section we give some results which allow us to get estimates for the $E_{n}(f)$ and are needed for constructing "near-best" polynomials. 

For $a>0$ and $b>0$ we will use the notation $a\preccurlyeq b$ if $a\leqslant cb$, with some constant $c>0$. The expression $a\asymp b$ means $a\preccurlyeq b$ and $b\preccurlyeq a$.

Let $\mathcal{L}$ be a quasi-smooth arc and $\Omega := \overline{\textbf{C}}\backslash  \mathcal{L}$.
Consider a conformal mapping
$\Phi: \Omega \longrightarrow \Delta := \{\omega: |\omega|>1\}$, normalized in such a way that $\Phi(\infty)=\infty$, $\Phi'(\infty)>0$,
and denote $\Psi:= \Phi^{-1}$.

By $\widetilde{\Omega}$ we denote compactification of the domain  $\Omega$ by prime ends in the Caratheodory sense, and  $\widetilde{\mathcal{L}}$ := $\widetilde{\Omega}\setminus \Omega$.
For the endpoints $z_{1}, z_{2}$ of $\mathcal{L}$ and $u>0, j=1,2$,  let
$$\Phi(z_{j}):=\tau_{j};
$$ $$
\Delta_{1}:=\{\tau: \tau \in \Delta, arg \: \tau_{1} < arg \:\tau < arg \: \tau_{2}\};
$$ $$
\Delta_{2}:=\Delta\backslash\overline{\Delta}_{1}, \;
\widetilde{\Omega}^{j}:=\Psi(\overline{\Delta}_{j}), \; \Omega^{j}:=\Psi(\Delta_{j}); \;
$$
$$
\widetilde{\mathcal{L}}^{j}:=\widetilde{\Omega}^{j} \cap \widetilde{\mathcal{L}};
$$ $$
\mathcal{L}^{j}_{u}:=\{\zeta: \zeta \in \widetilde{\Omega}^{j}, |\Phi(\zeta)|=1+u\};
$$ $$
\rho^{j}_{u}(z):=dist (z,\mathcal{L}^{j}_{u}); \;
\rho^{*}_{u}(z):=\underset{j=1,2}{max} \:\rho^{j}_{u}(z).
$$

Let $z_{0}$ be a point of $\mathcal{L}$, distinct from endpoints of the arc. Then point $z_{0}$ divides $\mathcal{L}$ into two parts, $\mathcal{L'}$ and $\mathcal{L''}$.
Consider the function 
$$ 
f(z)=\begin{cases}                                                              
f_{1}(z),&\text{if $z$ $\in \mathcal{L} ^{'}$}\\                             
f_{2}(z),&\text{if $z$ $\in \mathcal{L} ^{''}$}
\end{cases}                                                                   \eqno(2.1)
$$
where $f_{1}$, $f_{2}$ are functions, analytic on $\mathcal{L} ^{'}$ and $\mathcal{L} ^{''}$, i.e. analytic in some neighborhoods of $\mathcal{L} ^{'}$ and $\mathcal{L} ^{''}$ correspondingly, and satisfying
$$
f_{1}^{(r)}(z_{0})=f_{2}^{(r)}(z_{0})  , \; r=\overline{0,k},    \quad f_{1}^{(k+1)}(z_{0})\neq f_{2}^{(k+1)}(z_{0}).        \eqno(2.2)
$$

By $U$ we will denote an open circular neighborhood of the point $z_{0}$, where both $f_{1}$, $f_{2}$ are analytic.

Let $Z_{0}^{1}$, $Z_{0}^{2}$ $ \in \widetilde{\mathcal{L}}$ be the prime ends, s.t.  $|Z_{0}^{j}|=z_{0}, \; j=1,2$. Set
$$
\tau_{0}^{j}:=\Phi(Z_{0}^{j}), \quad j=1,2.
$$

Points $\tau_{1}^{j}, j=1,2$ we define by
$$
\tau_{1}^{j}=\lambda \tau_{0}^{j},
$$
with $\lambda>1$ such that
 $$ 
 Г^{1}, Г^{2} \subset U,
 $$
where
$$
Г^{j} = Г^{j}_{0}:=\{\zeta:  1<|\Phi(\zeta)|<\lambda, arg \; \Phi(\zeta)=arg\; \tau_{0}^{j}\}, \; j=1,2.                                                \eqno(2.3)
$$

The arcs $Г^{1},Г^{2}$ are rectifiable (see \cite[Chap. 5]{ConstrTheory}), thus, can be oriented in such a way that for all $z \in \mathcal{L}\backslash\{z_{0}\}$
function $f$ can be represented, by the Cauchy formula, as
$$
f(z) = h_{1}(z) + h_{2}(z),                                          
$$
where
$$
h_{1}(z)=\frac{1}{2\pi i} 
\int\limits_{Г^{1} \cup Г^{2}} \frac{f_{1}(\zeta)-f_{2}(\zeta)}{\zeta-z} d\zeta,  \eqno(2.4) 
$$
and $h_{2}(z)$ is analytic for all $z \in \mathcal{L}$, therefore it can be approximated with a geometric rate on $\mathcal{L}$.

We will make use of the following lemma. 

\begin{lemma}
Let $\mathcal{L}$ be a quasi-smooth arc. Then for any fixed non-negative integer $k$,  a positive integer $n$ and $\zeta$ $\in$ $Г^{1} \cup Г^{2}$ there exists a polynomial kernel $K_{n}(\zeta,z)$ of the form
$K_{n}(\zeta,z)=\displaystyle\sum_{j=0}^{n}a_{j}(\zeta)z^{j}$ with continuous in $\zeta$ 
coefficients $a_{j}$($\zeta$), $j=\overline{0,n}$,
satisfying for z$\; \in\mathcal{L}$ and $\zeta$ with $|\zeta - z_{0}| \geq        \rho_{1/n}^{*}(z_{0})$
$$	
|\frac{1}{\zeta-z} - K_{n}(\zeta,z)| \leq c [\rho_{1/n}^{*}(z_{0})]^{k+2} |\zeta - z_{0}|^{-(k+3)} ,  \eqno(2.5)
$$
where $c=c(\mathcal{L})>0$.
\end{lemma}

\begin{proof}
To show (2.5), we repeat word by word the proof for $k=0$, (\cite[Lemma 5.4]{ConstrTheory}).

Let $n$ be sufficiently large. For fixed $m$ and $r$ we consider the Dzyadyk polynomial kernel $K_{0,m,r,n}(\zeta, z)$ (see, e.g., \cite[Chap. 3]{ConstrTheory}). Then,
for $r\geqslant 5$ and $z \in \mathcal{L}$, $\zeta \in Г^{j}, j=1,2$,
$$
\bigg|\frac{1}{\zeta - z} -  K_{0,m,r,n}(\zeta, z)\bigg| \preccurlyeq \frac{1}{|\zeta-z|} \bigg| \frac{\widetilde{\zeta} - \zeta}{\widetilde{\zeta} - z}\bigg|^{rm}
$$
where $\widetilde{\zeta} := \widetilde{\zeta}^{j}_{1/n}:=\Psi [(1+1/n)\Phi(\zeta)]$. 

Since
$$
\bigg|\frac{\widetilde{\zeta} - \zeta}{\widetilde{\zeta} - z}\bigg| \preccurlyeq  \bigg|\frac{\widetilde{\zeta} - \zeta}{\zeta - z_{0}} \bigg| \preccurlyeq \bigg|\frac{\rho^{j}_{1/n}(z_{0})}{\zeta - z_{0}}\bigg|^{c} \preccurlyeq \bigg|\frac{\rho^{*}_{1/n}(z_{0})}{\zeta - z_{0}} \bigg|^{c},
$$

it is enough to take $r$ and $m$ such that $rmc\geqslant k+2$, and set $K_{n}(\zeta,z):= K_{0,m,r,[\varepsilon n]}(\zeta,z)$, where $\varepsilon = \varepsilon(r,m)>0$ is sufficiently small.  
\end{proof}

The next theorem is also a generalization of the case $k=0$ in (2.2) and the proof essentially repeats the proof of \cite[Theorem 5.2]{ConstrTheory}. 

\begin{theorem}
Let $\mathcal{L}$ be a quasi-smooth arc, and let function $\mathit{f}$ be given by (2.1), (2.2).
Then 
$$
c'\;[\rho_{1/n}^{*}(z_{0})]^{k+1} \leq
E_{n}(\mathit f, \mathcal L) \leq
c''\;[\rho_{1/n}^{*}(z_{0})]^{k+1},                     \eqno(2.6)
$$
where $c', c''$ don't depend on n.
\end{theorem}

\begin{proof}
First, we estimate $E_{n}(\mathit f, \mathcal L)$ from above.

Without loss of generality, we can assume $z_{0}=0$ and $n$ is sufficiently large. \\
Let $d_{n}:=\rho_{1/n}^{*}(0)$,
$\gamma=\gamma_{n}:=\{\zeta: \zeta \in Г^{1}\cup Г^{2}, |\zeta|\geq d_{n}\}$,
$$
P_{n}= \frac{1}{2\pi i}  \int_{\gamma} (f_{1}(\zeta)-f_{2}({\zeta})) K_{n}(\zeta, z) d\zeta. 
$$
From (2.2), for all $\zeta$ in some neighborhood $U$ of the point $z_{0} = 0$ 
$$
f_{1}(\zeta)=c_{0}+c_{1}\zeta+...+c_{k}\zeta^{k}+c_{k+1}\zeta^{k+1}+\varphi_{1}(\zeta) \zeta^{k+2}                          \eqno(2.7)
$$
$$
f_{2}(\zeta)=c_{0}+c_{1}\zeta+...+c_{k}\zeta^{k}+\widetilde{c}_{k+1}\zeta^{k+1}+\varphi_{2}(\zeta) \zeta^{k+2},                   \eqno(2.8)
$$
where $c_{k+1}\neq\widetilde{c}_{k+1}$ and $\varphi_{1}(\zeta),\varphi_{2}(\zeta)$ are functions, analytic in $U$.

Hence, there exists a constant $C$ such that
$$
|f_{1}(\zeta)-f_{2}({\zeta})| \leq C|\zeta^{k+1}|,  \qquad  \zeta \in U.        \eqno(2.9)                                      
$$

By (2.5), (2.9), for all $z \in \mathcal L$ 
$$
\bigg|\frac{1}{2\pi i} \int_{Г^{1} \cup Г^{2}} \frac{f_{1}(\zeta)-f_{2}({\zeta})}{\zeta-z} d\zeta - P_{n}(z)\bigg|
$$
$$
\leq \frac{1}{2\pi }  \int_{\gamma} |f_{1}(\zeta)-f_{2}({\zeta})|  \bigg| \frac{1}{\zeta-z} - K_{n}(z,\zeta)\bigg|  |d\zeta| 
+
\frac{1}{2\pi}  \int_{(Г^{1} \cup Г^{2})\backslash\gamma} \bigg| \frac{f_{1}(\zeta)-f_{2}({\zeta})}{\zeta-z}\bigg| |d\zeta| 
$$
$$
\leq \frac{C d_{n}^{k+2}}{2\pi }  \int_{\gamma} \frac{|d\zeta|}{|\zeta|^{2}}  
+
\frac{C}{2\pi}  \int_{(Г^{1} \cup Г^{2})\backslash\gamma}  \frac{|\zeta^{k+1}|}{|\zeta-z|} |d\zeta|.                                                              \eqno(2.10)
$$

Integration by parts of $\int_{\gamma} \frac{|d\zeta|}{|\zeta|^{2}}$ yields
$\int_{\gamma} \frac{|d\zeta|}{|\zeta|^{2}} \preccurlyeq \frac{1}{d_{n}} $.
Since $dist(\zeta,  \mathcal L) \asymp |\zeta|$, (see \cite[Chap. 5]{ConstrTheory}), 
and $|{(Г^{1} \cup Г^{2})\backslash\gamma}|\preccurlyeq d_{n}$, it implies $\int_{(Г^{1} \cup Г^{2})\backslash\gamma}  \frac{|\zeta^{k+1}|}{|\zeta-z|} |d\zeta|  \preccurlyeq d_{n}^{k+1}$. Thus, combining with (2.10), we obtain the estimate from above in (2.6).

Now, we estimate $E_{n}(\mathit f, \mathcal L)$ from below.

Let $p_{n}^{*}$ be the polynomial of the best approximation, that is
$$
|f(z)-p_{n}^{*}(z)| \leq E_{n}(f),  \qquad  z \in \mathcal L     \eqno(2.11)
$$

Without loss of generality we can assume that
$$
E_{n}(f) \leq d_{n} = \rho _{1/n}^{1}(0).
$$

Denote by $l_{3} \subset \Omega^{1}$ any arc of a circle $\{\zeta: |\zeta|=d_{n}\}$, 
separating the prime end $Z_{0}^{1}$ from $\infty$.

Let $ z' \in  \mathcal L'$ and  $z'' \in  \mathcal L''$ be the endpoints of the arc $l_{3}$.
Denote
$$
l_{1}:=\mathcal L(0,z'), \quad
l_{2}:=\mathcal L(0,z'').
$$

Next, take a point $z$ so that $z \in  Г^{2}$, $|z|=\varepsilon d_{n}$ (we'll choose the constant $\varepsilon$ later).
With a corresponding choice of orientation of arcs $l_{j}$, $j=1,2,3$
$$
I:= \int_{l_{1} \cup l_{2}} \frac{\widetilde{f}(\zeta)}{(\zeta-z)^{k+2}} d\zeta =  \int_{l_{1} \cup l_{2}} \frac{\widetilde{f}(\zeta)-\widetilde{p}_{n}^{*}(\zeta)}{(\zeta-z)^{k+2}} d\zeta  +    \int_{l_{3}} \frac{\widetilde{p}_{n}^{*}(\zeta)}{(\zeta-z)^{k+2}} d\zeta,    \eqno(2.12)
$$

where $\widetilde{f}(\zeta) = f(\zeta) - (c_{0}+c_{1}\zeta+...+c_{k}\zeta^{k})$
and $\widetilde{p}_{n}^{*}(\zeta) = p_{n}^{*} - (c_{0}+c_{1}\zeta+...+c_{k}\zeta^{k}) $.
Notice that $f(\zeta) - p_{n}^{*}(\zeta) = \widetilde{f}(\zeta) - \widetilde{p}_{n}^{*}(\zeta)$.

In the following estimates we use notations $a_{i}, \widetilde{a}_{i}, \widetilde{C}, \widehat{C}, C_{i}$ for constants. 

For the left hand side we have
$$
\bigg| \int_{l_{1} \cup l_{2}} \frac{\widetilde{f}(\zeta)}{(\zeta-z)^{k+2}} d\zeta \bigg|
$$

$$
= \bigg| c_{k+1} \int_{l_{1}} \frac{\zeta^{k+1}}{(\zeta-z)^{k+2}} d\zeta + \widetilde{c}_{k+1} \int_{l_{2}} \frac{\zeta^{k+1}}{(\zeta-z)^{k+2}} d\zeta 
$$
$$
+ \int_{l_{1}} \frac{\varphi_{1}(\zeta) \zeta^{k+2}   }{(\zeta-z)^{k+2}} d\zeta
+
\int_{l_{2}} \frac{\varphi_{2}(\zeta) \zeta^{k+2}   }{(\zeta-z)^{k+2}} d\zeta \bigg| 
$$
$$
= \bigg| c_{k+1} \log \frac{z}{z-z'} + \widetilde{c}_{k+1} \log \frac{z-z''}{z} + \frac{a_{1}z'^{k+1}+a_{2}z'^{k}z+ ...+a_{k+1}z'z^{k}}{(z' - z)^{k+1}} 
$$
$$
+\frac{\widetilde{a}_{1}z''^{k+1}+\widetilde{a}_{2}z''^{k}z+ ...+\widetilde{a}_{k+1}z''z^{k}}{(z'' - z)^{k+1}} + \widetilde{C} + \int_{l_{1}} \frac{\varphi_{1}(\zeta) \zeta^{k+2}   }{(\zeta-z)^{k+2}} d\zeta
+
\int_{l_{2}} \frac{\varphi_{2}(\zeta) \zeta^{k+2}   }{(\zeta-z)^{k+2}} d\zeta \bigg| 
$$
$$
\geqslant \bigg| (\widetilde{c}_{k+1} - c_{k+1}) \log \frac{z-z''}{z} + c_{k+1} \log \frac{z-z''}{z-z'}\bigg| - \frac{C_{1} \varepsilon}{(1 - \varepsilon)^{k+1}} - \widehat{C} 
$$
$$
\geqslant  |\widetilde{c}_{k+1} - c_{k+1}|  \log\frac{1-\varepsilon}{\varepsilon} - \frac{C_{1} \varepsilon}{(1 - \varepsilon)^{k+1}} - C_{2} 
$$
Next, we estimate the right hand side of (2.12). By (2.11) and by the choice of $z$ 

$$\bigg| \int_{l_{1} \cup l_{2}} \frac{\widetilde{f}(\zeta)-\widetilde{p}_{n}^{*}(\zeta)}{(\zeta-z)^{k+2}} d\zeta \bigg|  \leq C_{3} \frac{E_{n}}{\varepsilon^{k+1} d_{n}^{k+1}}.
$$ 
To estimate the integral over $l_{3}$ notice that by (2.7) and (2.8)
$$
|\widetilde{f}(\zeta)| \leq c|\zeta^{k+1}|,    \qquad \zeta \in  \mathcal{L}
$$
for some constant $c$. Without loss of generality, we assume $c=1$ (otherwise the arc $l_{3}$ must be considered with a radius $\frac{d_{n}}{c}$ instead).
Since the estimate 
$$
|\widetilde{p}_{n}^{*}(\zeta)| \leq |\widetilde{p}_{n}^{*}(\zeta) - \widetilde{f}(\zeta)| + |\widetilde{f}(\zeta)| \leq d_{n}^{k+1} \bigg(1 + \bigg|\frac{ \zeta}{ d_{n}}\bigg|^{k+1}\bigg),  \quad \zeta \in  \mathcal{L}
$$
holds, \cite[Theorem 6.1]{ConstrTheory} implies
$$
|p_{n}^{*}(\zeta)| \leq C_{4} d_{n}^{k+1} , \qquad  \zeta \in l_{3}.
$$
The last inequality yields
$$\bigg| \int_{l_{3}} \frac{\widetilde{p}_{n}^{*}(\zeta)}{(\zeta-z)^{k+2}} d\zeta \bigg| \leq \frac{2\pi C_{4}}{(1-\varepsilon)^{k+2}}.
$$
Combining the estimates above, for some small but fixed $\varepsilon$ we get
$$
C_{3} \frac{E_{n}}{\varepsilon^{k+1} d_{n}^{k+1}} \geqslant  |\widetilde{c}_{k+1} - c_{k+1}|  \log\frac{1-\varepsilon}{\varepsilon} - \frac{C_{1} \varepsilon}{(1 - \varepsilon)^{k+1}} - C_{2}  -  \frac{2\pi C_{4}}{(1-\varepsilon)^{k+2}} 
$$
$$
\geqslant \frac{|\widetilde{c}_{k+1} - c_{k+1}|}{2}   \log\frac{1-\varepsilon}{\varepsilon}.
$$
Consequently, the estimate from below in (2.6) holds.  

\end{proof}
With reasoning completely similar, we obtain the following.

\begin{theorem}
	Let $\mathcal{L}$ be a quasi-smooth arc, and let function $\mathit{f}$ be given by (1.3), (1.4).
	Then 
	$$
	c'\;[\rho_{1/n}^{*}(z_{0})]^{k+1} \leq
	E_{n}(\mathit f, \mathcal L) \leq
	c''\;[\rho_{1/n}^{*}(z_{0})]^{k+1},                   
	$$
	where $k := \underset{i=\overline{2,m-1}}{min} \{k_{i}\}$ and $c', c''$ don't depend on n.
\end{theorem}

\section{ Proof of Theorem 1}

As it was mentioned above, 
$f$ can be represented as
$$
f(z) = \sum_{j=2}^{m-1} (h^{j}_{1}(z) + h^{j}_{2}(z)),
$$
where $h^{j}_{2}(z)$ are analytic functions on $\mathcal{L}$
and 
$$
h^{j}_{1}(z) =\frac{1}{2\pi i} 
\int\limits_{Г_{j}^{1} \cup Г_{j}^{2}} \frac{f_{j-1}(\zeta)-f_{j}(\zeta)}{\zeta-z} d\zeta,
$$
with $Г_{j}^{1}, Г_{j}^{2}$ being the arcs given by (2.3), that correspond to the point $z_{j}$.
Therefore, it's enough to construct polynomial approximants for $h^{j}_{1}(z)$ only.

To approximate the integral over  $Г_{j}^{i}$, $i=1,2$, consider a function $F^{i}_{j}: \mathcal{L} \cup Г_{j}^{i} \longrightarrow \mathcal{L}_{\varphi, j}^{i} $,
such that $F^{i}_{j}$ is one-to-one and satisfies 
$$
|F^{i}_{j}(z)-F^{i}_{j}(\zeta)| \leqslant c |z-\zeta|,    \quad z, \; \zeta \in \mathcal{L} \cup Г_{j}^{i},
$$
$$
F^{i}_{j}(z_{j}) = 0,
$$
$$
F^{i}_{j}(\mathcal{L}(z_{1}, z_{j})) = \mathcal{L'}, 
$$
$$
F^{i}_{j}(\mathcal{L}(z_{j}, z_{m}))= \mathcal{L''},
$$
$$
F^{i}_{j}(Г_{j}^{i})=\widetilde{Г},
$$
where $z_{1}$, $z_{m}$ are endpoints of $\mathcal{L}$, $\mathcal{L'}$ is a line segment in $[0,\infty)$, $\mathcal{L''}$ is a line segment in the upper half plane that form an anle $\varphi>0$ with $\mathcal{L'}$, (this angle will be determined below), and $\widetilde{Г}$ -- a line segment at an angle $\frac{\varphi}{2}$ to the $\mathcal{L'}$.

Such a mapping $F^{i}_{j} \in Lip_{1}[\mathcal{L} \cup Г_{j}^{i}]$ always exists, 
and to see this it is enough to note that $\mathcal{L}$ and $Г_{j}^{i}$ are quasi-smooth and 
 $$dist(\zeta, \mathcal{L}) \asymp |\zeta - z_{j}|$$
 holds for all $\zeta \in Г_{j}^{i}$ (see \cite[Chap. 5]{ConstrTheory}).
 
\begin{figure}[h!]
	\center{\includegraphics[scale=0.5]{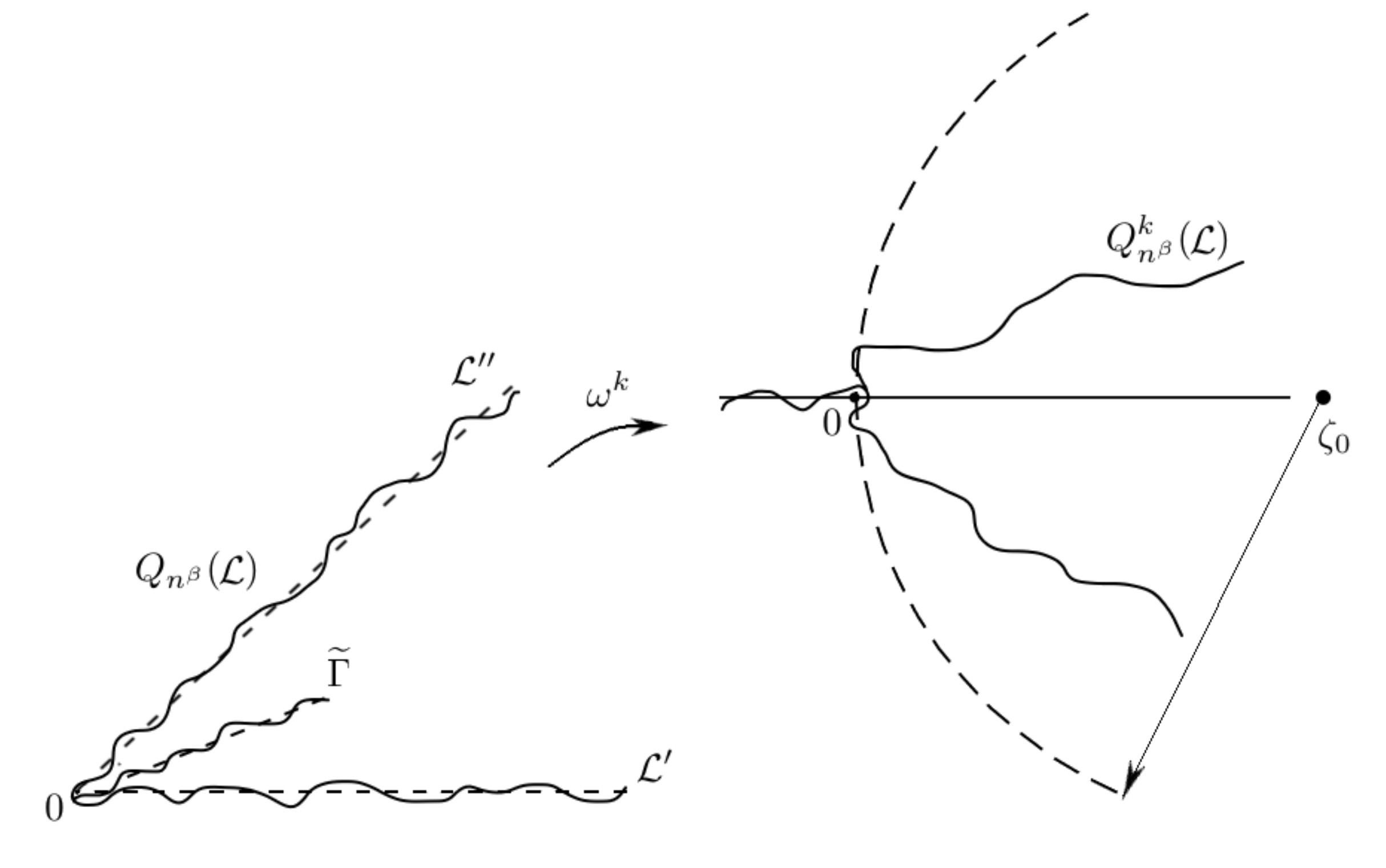}}
\end{figure}

By \cite[Theorem 4]{AppOnCont} the function $F^{i}_{j}$ can be approximated by polynomials $Q_{n}(z) := Q^{i}_{n,j}(z)$ with the rate $\frac{1}{n^{\alpha}}$, for some $\alpha >0$, that is
$$
|F^{i}_{j}(z) - Q_{n}(z)| \leqslant \frac{C}{n^{\alpha}},  \quad  z \in \mathcal{L} \cup Г_{j}^{i},                                                                  \eqno(3.1)
$$
where constant $C$ does not depend on $z$ and $n$.

For fixed $0<\sigma< 1$ take an integer $k\geqslant2 $, such that $1 - \sigma > \frac{1}{1+ k \alpha}$.
Now, for $\varphi = \frac{2\pi}{k}$ consider corresponding mapping $F^{i}_{j}$ and approximating polynomials $Q_{n}$. 

Let
$$
\widehat{P}^{i}_{n,j}(z,\zeta)=\frac{1-\left(\frac{Q^{k}_{[n^{\beta}]} (z)-\zeta_{0}}{Q^{k}_{[n^{\beta}]}(\zeta)-\zeta_{0}}\right)^{\big[\frac{n^{1-\beta}}{2k}\big]}}{\zeta-z}
+\left(\frac{Q^{k}_{[n^{\beta}]}(z)-\zeta_{0}}{Q^{k}_{[n^{\beta}]}(\zeta)-\zeta_{0}}\right)^{\big[\frac{n^{1-\beta}}{2k}\big]} \; K_{[\frac{n}{2}]}(z,\zeta).
$$
It is not hard to see that $\widehat{P}^{i}_{n,j}(z,\zeta)$ is a polynomial in $z$ of degree at most $n$. The idea of constructing such a polynomial is motivated by \cite[p. 380]{Varga}. 

We will show that for some choice of $\beta$ and $\zeta_{0}$ the term $\bigg|\frac{Q^{k}_{[n^{\beta}]} (z)-\zeta_{0}}{Q^{k}_{[n^{\beta}]}(\zeta)-\zeta_{0}}\bigg|^{\big[\frac{n^{1-\beta}}{2k}\big]}$ is bounded uniformly on $\mathcal{L}$ by a constant that does not depend on $n$, and at points of analyticity of $f$ it can be bounded by $q^{\big[\frac{n^{1-\beta}}{2k}\big]}$, for some $q<1$.

For $n$ sufficiently large the arc $\mathcal{L}$ can be written as a disjoint union
$$
\mathcal{L} = A_{1} \cup A_{2} \cup A_{3},
$$
where
$$
A_{1}:= \{z \in \mathcal{L}: |Q_{n}(z)| < \frac{C}{n^{\alpha} sin \frac{\pi}{4k}}\},       \eqno(3.2)
$$
$$
A_{2}:=  \{z \in \mathcal{L}: |Q_{n}(z)| \geqslant \frac{C}{n^{\alpha} sin \frac{\pi}{4k}}, \: dist(Q_{n}(z), \mathcal{L'}) \leqslant  \frac{C}{n^{\alpha}} \},                        \eqno(3.3)
$$
$$
A_{3}:= \{z \in \mathcal{L}: |Q_{n}(z)| \geqslant \frac{C}{n^{\alpha} sin \frac{\pi}{4k}}, \: dist(Q_{n}(z), \mathcal{L''}) \leqslant \frac{C}{n^{\alpha}}\},         \eqno(3.4)
$$
where $C$ is the constant from (3.1).\\
Points of $A_{1}$ satisfy
$$
|Q^{k}_{n}(z)|<  \frac{C^k}{n^{\alpha k} (sin \frac{\pi}{4k})^k}.   \eqno(3.5)
$$
For $A_{2}$ we have
$$
|sin(arg\: Q_{n}(z))| \leqslant \frac{C}{n^{\alpha} |Q_{n}(z)|} \leqslant sin \frac{\pi}{4k},
$$
that implies
$$
- \frac{\pi}{4} \leqslant arg\: Q^{k}_{n}(z) \leqslant \frac{\pi}{4}.    \eqno(3.6)
$$
Similarly, for $A_{3}$
$$
|sin(\frac{2\pi}{k} - arg\: Q_{n}(z))| \leqslant \frac{C}{n^{\alpha} |Q_{n}(z)|} \leqslant sin \frac{\pi}{4k},
$$
that yields
$$
- \frac{\pi}{4} \leqslant arg\: Q^{k}_{n}(z) \leqslant \frac{\pi}{4}.    \eqno(3.7)
$$
 $Г^{i}_{j}$ can also be written as a disjoint union
$$
Г^{i}_{j} = B_{1} \cup B_{2},
$$
where
$$
B_{1}:= \{\zeta \in Г^{i}_{j}: |Q_{n}(\zeta)| < \frac{C}{n^{\alpha} sin \frac{\pi}{4k}}\},  \eqno(3.8)
$$
$$
B_{2}:= \{\zeta \in Г^{i}_{j}: |Q_{n}(\zeta)| \geqslant \frac{C}{n^{\alpha} sin \frac{\pi}{4k}}\}.    \eqno(3.9)
$$ 
Points of $B_{1}$ satisfy
$$
|Q^{k}_{n}(\zeta)|<  \frac{C^k}{n^{\alpha k} (sin \frac{\pi}{4k})^k}.   \eqno(3.10)
$$
For $B_{2}$ we have
$$
|sin(\frac{\pi}{k} - arg\: Q_{n}(\zeta))| \leqslant \frac{C}{n^{\alpha} |Q_{n}(\zeta)|} \leqslant sin \frac{\pi}{4k},    
$$

$$
\pi - \frac{\pi}{4} \leqslant arg\: Q^{k}_{n}(\zeta) \leqslant \pi + \frac{\pi}{4}.
\eqno(3.11)
$$

Now, if we choose $\zeta_{0}$ to be a point in $(0, \infty)$ with $\zeta_{0}>max\{|\mathcal{L'}|^{k}, |\mathcal{L''}|^{k}\}$, then (3.5), (3.6) and (3.7) imply
$|Q^{k}_{n}(z) - \zeta_{0}| \leqslant \zeta_{0} +  \frac{C^k}{n^{\alpha k} (sin \frac{\pi}{4k})^k}$, $z \in \mathcal{L}$.
Also, by (3.10) and (3.11) the estimate $|Q^{k}_{n}(\zeta) - \zeta_{0}| \geqslant \zeta_{0} -  \frac{C^k}{n^{\alpha k} (sin \frac{\pi}{4k})^k}$ holds for $\zeta \in Г^{i}_{j}$.   

According to these observations, we have
$$
\bigg|\frac{Q^{k}_{[n^{\beta}]} (z)-\zeta_{0}}{Q^{k}_{[n^{\beta}]}(\zeta)-\zeta_{0}}\bigg|^{\big[\frac{n^{1-\beta}}{2k}\big]} \leqslant  \bigg(1 + \frac{\widetilde{C}}{n^{\alpha \beta k}}\bigg) ^{\big[\frac{n^{1-\beta}}{2k}\big]}     \eqno(3.12)
$$
where $\widetilde{C}= \frac{2C^{k}}{\zeta_{0}  n^{\alpha \beta k}(sin \frac{\pi}{4k})^k - C^{k}} \leqslant C^{k}$ for $n$ large enough.

Let $\beta$ be such that $1 - \sigma >  \beta > \frac{1}{1+ k \alpha}$, so that $1 - \beta < \alpha \beta k$ and $\sigma < 1 - \beta $.
From (3.12) it follows

$$
\bigg|\frac{Q^{k}_{[n^{\beta}]} (z)-\zeta_{0}}{Q^{k}_{[n^{\beta}]}(\zeta)-\zeta_{0}}\bigg|^{\big[\frac{n^{1-\beta}}{2k}\big]}  \leqslant e^{ \widetilde{c} n^{- \alpha \beta k} n^{1 - \beta}} \leqslant \widehat{C},     \eqno(3.13)
$$
where $\widehat{C} $ does not depend on $n$.

Also, for all points $z$ of a compact set $E \subset \mathcal{L} \backslash \{z_{1}, z_{2},...,z_{m}\}$ and $n$ sufficiently large the estimate 
$$
\bigg|\frac{Q^{k}_{[n^{\beta}]} (z)-\zeta_{0}}{Q^{k}_{[n^{\beta}]}(\zeta)-\zeta_{0}}\bigg|^{\big[\frac{n^{1-\beta}}{2k}\big]} \leqslant q^{\big[\frac{n^{1-\beta}}{2k}\big]},    \eqno(3.14)
$$
holds with some $q = q(E)<1$.

Therefore, if we denote

$d_{n}:=\rho_{1/n}^{*}(z_{j})$,
$\gamma=\gamma_{n}:=\{\zeta: \zeta \in Г^{i}_{j}, |\zeta-z_{j}|\geq d_{n}\}$  

and consider polynomial
$$
P^{i}_{n,j}(z)= \frac{1}{2\pi i}  \int_{\gamma} (f_{j-1}(\zeta)-f_{j}(\zeta))  \widehat{P}^{i}_{n,j}(z,\zeta) d\zeta 
$$
$$ + \frac{1}{2\pi i} \int_{Г^{i}_{j}\backslash\gamma}  (f_{j-1}(\zeta)-f_{j}(\zeta))   \frac{\left(1-\left(\frac{Q^{k}_{[n^{\beta}]}(z)-\zeta_{0}}{Q^{k}_{[n^{\beta}]}(\zeta)-\zeta_{0}}\right)^{\big[\frac{n^{1-\beta}}{2k}\big]}\right) }{\zeta-z} d\zeta,
$$
by (3.13) and Theorem 4, for all $z\in \mathcal{L}$ we get

$$
\bigg| \frac{1}{2\pi i} 
\int\limits_{Г_{j}^{i}} \frac{f_{j-1}(\zeta)-f_{j}(\zeta)}{\zeta-z} d\zeta - P^{i}_{n,j}(z) \bigg|               
$$
$$
\leqslant \frac{1}{2\pi }  \int_{\gamma} |f_{j-1}(\zeta)-f_{j}({\zeta})| \bigg| \frac{Q^{k}_{[n^{\beta}]}(z)-\zeta_{0}}{Q^{k}_{[n^{\beta}]}(\zeta)-\zeta_{0}} \bigg|^{\big[\frac{n^{1-\beta}}{2k}\big]}  \bigg| \frac{1}{\zeta-z} - K_{[\frac{n}{2}]}(z,\zeta)\bigg|  |d\zeta|
$$
$$
+ \frac{1}{2\pi}  \int_{Г^{i}_{j}\backslash\gamma} \bigg| \frac{f_{j-1}(\zeta)-f_{j}({\zeta})}{\zeta-z}\bigg|  \bigg| \frac{Q^{k}_{[n^{\beta}]}(z)-\zeta_{0}}{Q^{k}_{[n^{\beta}]}(\zeta)-\zeta_{0}} \bigg|^{\big[\frac{n^{1-\beta}}{2k}\big]}    |d\zeta| 
$$
$$
\preccurlyeq d_{n}^{k_{j}+2}  \int_{\gamma} \frac{|d\zeta|}{|\zeta|^{2}}    
+  \int_{Г^{i}_{j}\backslash\gamma} \bigg| \frac{\zeta^{k_{j}+1}}{\zeta-z}\bigg|      |d\zeta| \preccurlyeq
 E_{n}(f, \mathcal{L}),    \eqno(3.15)
$$
where the last inequality follows by the reasoning, similar to the one we use in (2.10).
 
If $z \in E$, by (3.14), (2.10) and Theorem 4 we have
$$
\bigg| \frac{1}{2\pi i} 
\int\limits_{Г_{j}^{i}} \frac{f_{j-1}(\zeta)-f_{j}(\zeta)}{\zeta-z} d\zeta - P^{i}_{n,j}(z) \bigg| 
$$
$$
\leq \frac{q^{\big[\frac{n^{1-\beta}}{2k}\big]} }{2\pi }  \int_{\gamma} |f_{j-1}(\zeta)-f_{j}({\zeta})| \bigg| \frac{1}{\zeta-z} - K_{[\frac{n}{2}]}(z,\zeta)\bigg|  |d\zeta| 
$$
$$
+ \frac{q^{\big[\frac{n^{1-\beta}}{2k}\big]} }{2\pi}  \int_{Г^{i}_{j} \backslash\gamma} \bigg| \frac{f_{j-1}(\zeta)-f_{j}({\zeta})}{\zeta-z}\bigg| |d\zeta|  
$$
$$
\preccurlyeq E_{n}(f, \mathcal{L}) q^{\big[\frac{n^{1-\beta}}{2k}\big]}   \preccurlyeq E_{n}(f, \mathcal{L}) e^{-\widetilde{c} n ^{1-\beta}}.                     \eqno(3.16)
$$

Let $P_{n}(z)=\Sigma_{j=2}^{m-1}(P^{1}_{n,j}(z)+P^{2}_{n,j}(z))$.

By (3.15), (3.16), polynomials $\{P_{n}\}$ are "near best"  polynomials, approximating  $\sum_{j=2}^{m-1} h^{j}_{1}(z)$ and satisfying (1.5).    $\quad \qedsymbol$

\section{ Proof of Theorem 2}

  Since changing the $R$ corresponds to scaling the lemniscate, we can always scale the picture and without loss of generality assume for simplicity $R=1$.

As it was shown above, it's enough to approximate the function
$$
h_{1}(z)=\frac{1}{2\pi i} 
\int\limits_{Г^{1} \cup Г^{2}} \frac{f_{1}(\zeta)-f_{2}(\zeta)}{\zeta-z} d\zeta. 
$$
Here $Г^{1}$ and $Г^{2}$ we choose in such a way that
$|P(\zeta)|>1$ for all  $ \zeta \in (Г^{1} \cup Г^{2}) \setminus \{0\}$.
While the image of  $Г^{1} \cup Г^{2}$ under the mapping $P$  belongs to the complement of the unit disc, the image of $\mathcal{L}$ is inside the disc, that yields
$$
\bigg|\frac{P(z)}{P(\zeta)}\bigg| \leqslant 1,  \qquad   z \in \mathcal{L}, \quad \zeta \in Г^{1} \cup Г^{2}      \eqno(4.1)
$$
Moreover, due to geometry of $\mathcal{L}$ the equality in (4.1) occurs only if $\zeta = z = 0$.

Let 
$$
\widehat{P}_{n}(z,\zeta)=\frac{1-\left(\frac{P(z)}{P(\zeta)}\right)^{[\frac{n}{2N}]}}{\zeta-z}
+\left(\frac{P(z)}{P(\zeta)}\right)^{[\frac{n}{2N}]}  K_{[\frac{n}{2}]}(z,\zeta).      \eqno(4.2)
$$

One may check that $\widehat{P}_{n}(z,\zeta)$ is a polynomial in $z$ of degree at most $n$. 

Let
$d_{n}:=\rho_{1/n}^{*}(0)$,
$\gamma=\gamma_{n}:=\{\zeta: \zeta \in Г^{1} \cup Г^{2}, |\zeta|\geq d_{n}\}$,
and 
consider
$$
P_{n}(z)= \frac{1}{2\pi i}  \int_{\gamma} (f_{1}(\zeta)-f_{2}(\zeta))  \widehat{P}_{n}(z,\zeta) d\zeta 
$$
$$
+  \frac{1}{2\pi i} \int_{(Г^{1} \cup Г^{2})\backslash\gamma}  (f_{1}(\zeta)-f_{2}(\zeta))  \left(\frac{1-\left(\frac{P(z)}{P(\zeta)}\right)^{[\frac{n}{2N}]}}{\zeta-z}\right) d\zeta.
$$

By virtue of Theorem 3, estimates (2.10) and (4.1), for all $z\in \mathcal{L}$ 
$$
\bigg|\frac{1}{2\pi i} \int_{Г^{1} \cup Г^{2}} \frac{f_{1}(\zeta)-f_{2}({\zeta})}{\zeta-z} d\zeta - P_{n}(z)\bigg|
$$
$$
\leq \frac{1}{2\pi }  \int_{\gamma} |f_{1}(\zeta)-f_{2}({\zeta})| \bigg| \frac{P(z)}{P(\zeta)} \bigg|^{[\frac{n}{2N}]}  \bigg| \frac{1}{\zeta-z} - K_{[\frac{n}{2}]}(z,\zeta)\bigg|  |d\zeta| 
$$
$$
+ \frac{1}{2\pi}  \int_{(Г^{1} \cup Г^{2})\backslash\gamma} \bigg| \frac{f_{1}(\zeta)-f_{2}({\zeta})}{\zeta-z}\bigg|  \bigg| \frac{P(z)}{P(\zeta)} \bigg|^{[\frac{n}{2N}]}  |d\zeta| 
$$
$$
\leq \frac{1}{2\pi }  \int_{\gamma} |f_{1}(\zeta)-f_{2}({\zeta})| \bigg| \frac{1}{\zeta-z} - K_{[\frac{n}{2}]}(z,\zeta)\bigg|  |d\zeta| 
+ \frac{1}{2\pi}  \int_{(Г^{1} \cup Г^{2})\backslash\gamma} \bigg| \frac{f_{1}(\zeta)-f_{2}({\zeta})}{\zeta-z}\bigg| |d\zeta| 
$$
$$
\preccurlyeq d_{n}^{k+2}  \int_{\gamma} \frac{|d\zeta|}{|\zeta|^{2}}    
+  \int_{(Г^{1} \cup Г^{2})\backslash\gamma} \bigg| \frac{\zeta^{k+1}}{\zeta-z}\bigg|      |d\zeta| 
$$
$$
\preceq [\rho^{*}_{1/n}(0)]^{k+1} \preccurlyeq  E_{n}(f, \mathcal{L}).
$$
If $E$ is a compact set in $\mathcal{L}\setminus \{z_{1},0,z_{2}\}$, then for all $ z \in  E$
$$
|P(z)|<q,     \eqno(4.3)
$$
for some $q=q(E)<1$.

Let
$$
d(E):=\underset{z\in E}{min}\{1-|P(z)|\}.
$$
By (4.3), $d(E)>0$ for any compact set $E \subset \mathcal{L}\setminus \{z_{1},0,z_{2}\}$.

Therefore, for all $z \in E$ 
$$
\bigg|\frac{P(z)}{P(\zeta)}\bigg|^{[\frac{n}{2N}]} \leq |P(z)|^{[\frac{n}{2N}]} \leq |1 - d(E)|^{[\frac{n}{2N}]} \leq e^{-c n d(E)},
$$
where the constant $c>0$ does not depend on $n$ and $E$. 

Hence, for $z \in E$
$$
\bigg|\frac{1}{2\pi i} \int_{Г^{1} \cup Г^{2}} \frac{f_{1}(\zeta)-f_{2}({\zeta})}{\zeta-z} d\zeta - P_{n}(z)\bigg|
$$
$$
\leq \frac{e^{-c n d(E)}}{2\pi }  \int_{\gamma} |f_{1}(\zeta)-f_{2}({\zeta})| \bigg| \frac{1}{\zeta-z} - K_{[\frac{n}{2}]}(z,\zeta)\bigg|  |d\zeta| 
$$
$$
+ \frac{e^{-c n d(E)}}{2\pi}  \int_{(Г^{1} \cup Г^{2})\backslash\gamma} \bigg| \frac{f_{1}(\zeta)-f_{2}({\zeta})}{\zeta-z}\bigg| |d\zeta| 
$$
$$
\preccurlyeq [\rho^{*}_{1/n}(0)]^{k+1}e^{-c n d(E)}  \preccurlyeq   E_{n}(f, \mathcal{L}) e^{-c n d(E)}.   \quad \qedsymbol
$$
 
$$
$$

\textbf{Acknowledgment}

The author would like to warmly thank Vladimir Andrievskii for guidance and many useful discussions.

\end{document}